\documentclass{article}
\usepackage[fleqn]{amsmath}
\usepackage{ntheorem-hyper}
\usepackage{amssymb}
\usepackage{authblk} 
\usepackage{algorithm}
\usepackage{algorithmicx}
\usepackage{algpseudocode}
\usepackage{graphicx}
\usepackage{graphbox}
\usepackage{subfigure}
\usepackage{svg}
\usepackage{epstopdf}
\usepackage{tabularx}  
\usepackage{booktabs}
\usepackage[numbers, sort]{natbib}
\usepackage{url}
\usepackage{hyperref}
\usepackage[top=2cm, bottom=2cm, left=2cm, right=2cm]{geometry}  
\usepackage{xcolor} 
\usepackage{color}
\usepackage[title]{appendix}
\hypersetup{urlcolor=blue, citecolor=red}

\newtheorem{theorem}{Theorem}
\newtheorem{prop}{Proposition}
\theorembodyfont{\upshape}
\newtheorem*{proof}{Proof}
\newtheorem{remark}{Remark}[section]
\newtheorem{corollary}{Corollary}
\newtheorem{lemma}{Lemma}

\newtheorem{definition}{Definition}[section]

\newtheorem{problem}{Problem}
\newtheorem{claim}{Claim}
\newtheorem*{theoremnonumber}{Theorem}

\def\RR{\mathbb{R}}
\def\SS{\mathbb{S}}
\def\TT{\mathbb{T}}

\title{On equivariant isometric embeddings of Riemannian manifolds with symmetries}
\author[+]{Hongda Qiu \\ Email: hjq5042@psu.edu}
\begin{document}
\maketitle
\interfootnotelinepenalty=10000 
\begin{abstract}
	Let $(M,g)$ be a $C^\infty$-smooth, $n$-dimensional Riemannian manifold which is diffeomorphic to $\RR^n$ and admit an action of a properly discontinuous and cocompact group. This work proves the existence of a $C^\infty$ equivariant isometric embedding of $M$ in some Euclidean space $\RR^q$ where $q = \max\{s_n+2n, s_n+n+5\}$ is the same as the dimension of Matthias Günther's results.
\end{abstract}
\section{Introduction}
\textit{History.} The isometric embedding problem of Riemannian manifolds into Euclidean spaces has a rich history. John Nash \cite{nash1956} proved the existence of a global isometric embedding for every smooth Riemannian manifold into some high dimensional Euclidean space. Later, Gromov, Rokhlin and Günther \cite{rokhlin1968}\cite{Gromov1970}\cite{Gunther1989} enhanced Nash's result by reducing the required dimension of target space and by reformulating the problem in a language more accessible for further exploitation. Specifically, they directly simplified Nash's original scheme and achieved the best required dimension known so far. Moreover, in the results above, the embedding has the same regularity as the metric of the manifold. In this work, we work with $C^\infty$-case, which is consistent with Günther's statements.


As the study of isometric embeddings deepened, more attention has expanded besides smoothness and required dimension. One novel question is: Can isometric embeddings preserve symmetry? That is, if a Riemannian manifold is acted upon by a group, can we find an isometric embedding that is \textit{equivariant}, meaning that the embedding respects this group? Furthermore, for which type of groups is this possible?

Equivariant isometric embeddings of Riemannian manifolds have been studied historically for specific examples such as compact homogeneous spaces \cite{Moore1976}, compact manifolds acted upon by a Lie group \cite{Moore1980}\cite{mostwo1956} or a finite group \cite{wang2022}, and symmetric spaces \cite{cartan1929}\cite{Eschenburg2022}\cite{clozel2007}. Nevertheless, the schemes of Nash and Günther are somehow used as a black box, and it remains an open problem whether their constructions can preserve symmetry of general manifolds, especially non-compact cases. This motivates an investigation into approaches that engage more with the analytic structure of their methods. 

\textit{Focus of study.} We study smooth equivariant isometric embeddings of manifolds with symmetries, specifically, the universal covers of Riemannian tori, and establish an existence theorem. At first glance, one might consider applying the Nash embedding theorem separately to every fundamental domain of the universal cover, and then gluing the embedded images together along their boundaries. However, this is flawed, as the images of adjacent fundamental domains do not necessarily align along their shared boundaries.

Our idea is to modify Günther's method into an ``equivariant" version, yielding a result that does not incur extra cost in the required dimension of target space. Understanding his approach contextualizes our approach. Günther's machine is summarized in the next section, followed by a presentation of our modification in the one after. Our language combines the reformulation by Han and Hong \cite[Section~1.2-1.3]{han2006isometric} and Yang \cite{yang1998gunthersproofnashsisometric}.

\textit{Günther's machine.} The isometric embedding problem of Riemannian manifolds is stated as follows: Given a Riemannian manifold $(M^n,g)$, does there exist a smooth immersion/embedding $u:M^n\to\RR^N$, for $ N $ explicitly determined by $n$, such that the metric induced on $M$ via $u$ is ${g}$; that is, $du\cdot du = g$? \footnote{For convenience, in the context of isometric embedding problems, the metric tensor and its associated differential form are not distinguished.}  

Both Nash's and Günther's resolutions to the problem above involve solving a local perturbation problem formulated as follows. Let $U$ be a simply-connected, bounded open subset of $M$, and $q$ a positive integer to be specified later. Given a smooth embedding $ w:M\to\RR^q $ (NOT necessarily isometric) with an important condition (Definition~\ref{definition free map}) to be specified shortly, and a symmetric $2$-tensor field $h$ whose support is contained in $U$ (which is regarded as the goal of local perturbation), find a smooth map $ v:U \to \RR^q $ such that
\begin{equation}
	\label{eq basic}
	d(w+v)\cdot d(w+v) = dw\cdot dw + h,
\end{equation}
where ``$\cdot$" denotes the inner product in Euclidean space.

In local coordinates on $M$ within $ U $, System~\ref{eq basic} consists of $\frac{1}{2}n(n+1)$ equations. Denote $s_n:=\frac{1}{2}n(n+1)$. Following Nash and Günther, we consider the case when perturbations represented by $v$ are normal to $u(M)$ in $\RR^q$; that is, $\partial_i{w}\cdot v = 0$. Write $h = \mathop{\sum}\limits_{i,j}h_{ij}dx_idx_j$. The system under consideration, after some direct computation, becomes:
\begin{equation}
	\label{eq s_n+n}
	\begin{cases}
		&\partial_i{w}\cdot v = 0,\\
		&\partial_i\partial_j{w}\cdot v =-\frac{1}{2}h_{ij} + \frac{1}{2}\partial_iv\cdot\partial_jv,\ \text{for}\ i,j=1,2,...,n.
	\end{cases}
\end{equation}
We want the system above, when formally regarded as a linear system in the unknown vector $v$, to be irreducible. To ensure this, $q$ must be at least $s_n+n$, and an additional condition on $w$ is required. 
\begin{definition}[Free map]
	\label{definition free map}
	We say that a map is \textit{free} if its first and second order partial derivatives are linearly independent at every point of its domain. \footnote{This definition is originally proposed by Gromov and Rokhlin \cite{Gromov1970}.}
\end{definition}
Assuming that $w$ is free, Günther proved the local solvability of System \ref{eq s_n+n} \cite[Theorem~2, p.~1140]{Gunther1990}. \footnote{This is also a step in Nash's work, referred to as the Nash-Moser implicit function theorem. However, the original scheme of Nash is explained and proved in a highly complicated manner. Günther's theorem plays the same role with a lower dimension.} 


The machine employed by Günther consists of two stages. The first one is a preparation for the second: Given $(M^n,g)$, construct a short free embedding $u_*:M\to\RR^{s_n+2n}$. The second stage is a process during which a short free embedding of $M$ is perturbed (through a sequence of free embeddings) to be an isometric one. With a localizing argument (\cite[Statement~2.2, Section~2, p.~168]{Gunther1989}\footnote{For English, see \cite[Lemma~1.3.1, p.~16]{han2006isometric}}.), the metric of $M$ is rearranged as the sum of a collection of symmetric $2$-tensor fields whose supports are contained in open subsets from an open cover of $M$. This open cover is chosen to be locally finite, so the perturbing process can be performed separately on each of these open subsets. Applying this process, it is shown (\cite[Theorem~3, p.~1141]{Gunther1990}) that every short free embedding of $ M $ into $\RR^q $, with\footnote{This dimension bound on $q$ is part of the assumption on the given embedding. To clarify, we use the term ``required dimension" when stating an embedding theorem of a certain type of manifold, and ``dimension bound" when stating a perturbing statement.} $q \geq s_n+ n + 5 $, can be modified to be isometric in an arbitrarily and uniformly small change of its image (small with respect to the Euclidean distance in $\RR^q$). Combining the two stages together and selecting the larger of $\RR^{s_n+2n}$ and $\RR^{s_n+n+5}$ as the target space, Günther re-established Nash's theorem with a required dimension of $\max\{s_n+2n,s_n+n+5\}$. 


\textit{Our approach.} We set up our study in higher generality. Instead of a Riemannian torus, we consider a Riemannian manifold $(M,g)$ that is diffeomorphic to $\RR^n$ and admits a Bieberbach group $\Gamma$ (by isometries). The quotient $M/\Gamma$ naturally defines a compact manifold covered by $M$. We modify Günther's machine as follows. 1. In Section~\ref{subsection initial embedding}, we show the existence of a free equivariant embedding from $M$ to $\RR^{s_n+2n}$ (Lemma~\ref{lemma initial}). This embedding plays the same role as the free embedding $u_*$ in the first stage of Günther's machine. 2. In Section~\ref{subsection main theorem}, a periodic open cover of $M$ is chosen by translating the pre-images of open sets of a finite open cover of ${M/\Gamma}$ by deck transformations. Every local perturbation is performed periodically, with its support contained in one set of this open cover. 




Finally, by modifying Günther's perturbing process and proving an argument to avoid self-intersections, we establish our main result (Theorem~\ref{theorem main}) in Section~\ref{subsection main theorem}. 

\begin{theorem}
	\label{theorem main}
	Let $(M,g)$ be a Riemannian manifold which is diffeomorphic to $\RR^n$ and admits an action of a Bieberbach group $\Gamma$. Let $ u_0:M\to\RR^q, q\geq s_n+n+5 $ be a smooth free equivariant embedding such that $g-du_0\cdot du_0$ is positive-definite. Then there exists a smooth free equivariant isometric embedding $u:M\to\RR^{q}$. 
\end{theorem}
Applying Lemma~\ref{lemma initial} in Section~\ref{subsection initial embedding},
\begin{corollary}
	\label{corollary main}
	$(M,g)$ possesses a smooth equivariant isometric embedding into $\RR^q$ with $q={\max\{s_n+2n, s_n+n+5\}}$.
\end{corollary}
Our approach illustrates that it is possible to impose symmetry-preserving conditions on isometric embeddings without increasing the required dimension when compared to Günther's original result. Compared to the joint work of Dmitri Burago and the author \cite[Theorem~2]{hongda2025embedding}, this result modifies Günther's machine directly. However, we need to reconsider if our approach can be generalized to the case when $\Gamma$ is not co-compact.

Finally, we propose and discuss several open problems.

\begin{problem}
	Let $M$ be a Riemannian manifold which is diffeomorphic to $\RR^n$ and admits a group action that is properly discontinuous but non-co-compact. Does $M$ admit a smooth equivariant isometric embedding into some Euclidean space? If the answer is positive, what would be the optimal dimension?
\end{problem}

We hope that there is a way to answer the following open problem by Gromov affirmatively:
\begin{problem}
	Is it possible to reduce the required dimension to $s_n+n$ for $C^\infty$-generic embeddings on the universal cover of a Riemannian torus? 
\end{problem}
One can also ask the problem above for a smooth manifold that is diffeomorphic to $\RR^n$ and admits a Bieberbach group action. A potential path to this problem seems to be by modifying Günther's Lemma (Lemma~\ref{lemma technical tool by gunther}). Namely, given $M$, one might try to prove an analogue of Lemma~\ref{lemma technical tool by gunther} for generic $C^\infty$-smooth metrics on $M$ with $q=s_n+n$ and establish a generic version of Theorem~\ref{theorem main}.

On the contrary, one can ask whether it is possible to reduce the dimension of target space $q$ to any amount smaller than $s_n+n$, but the answer is likely negative as System \ref{eq s_n+n} becomes over-determined. Furthermore, there are known results about non-embedibility (see, for instance, \cite[Theorem~1, Appendix~I]{Gromov1970}) suggesting that it is hard to reduce $q$ below $s_n-1$.

\textbf{Acknowledgement}: I would like to thank Dmitri Burago for his time on the terrible first draft of this paper.
%
%
%




%
%
%
%

\section{Proofs}
\subsection{Initial embedding}
\label{subsection initial embedding}
\begin{lemma}
	\label{lemma initial}
	There exists a smooth free equivariant embedding $u_0:M\to\RR^{s_n+2n}$ such that $g-du_0\cdot du_0$ is positive-definite.
\end{lemma}
\begin{remark}
	This lemma is an analogue of \cite[Statement~2.5.3]{Gromov1970} in our settings. In our proof, we use a trick from the proof of this referred statement.
\end{remark}
\begin{proof}
	We claim the following:
	\begin{claim}
		\label{claim initial}
		There exists a smooth immersion ${w_0}:{M/\Gamma}\to\RR^{q}$, with $q=s_n+n$, such that its second order partial derivatives are linearly independent pointwise on the entire ${M/\Gamma}$. 
	\end{claim}
	We first prove the lemma applying this claim, then prove the claim itself.
	
	\textbf{Proof of Lemma~\ref{lemma initial} applying Claim~\ref{claim initial}} Denote by $p:M\to{M/\Gamma}$ the covering map. Let us define
	\begin{equation}
		u_0: M\to\RR^{s_n+2n}, x\mapsto(x, {w_0}\circ p(x)).
	\end{equation}
	Since $x\mapsto x$ is an identity map, $u_0$ is injective, hence an embedding. Clearly, this embedding is equivariant. To see freeness, note that its derivatives are
	\begin{equation}
		\partial_iu_0 = (\underbrace{0,...,0,\overbrace{1}^{the\ i-th},0,...,0}_{n\ arguments},\quad\partial_i({w_0}\circ p)), \partial_i\partial_ju_0 = (\overrightarrow{0}, \partial_i\partial_j({w_0}\circ p))\tag*{for $1\leq i,j\leq n$.}
	\end{equation}
	Here, since the covering map $p$ is locally isometric, we can apply Claim~\ref{claim initial} to conclude that $ \partial_i\partial_j{(w_0\circ p)} $ are linearly independent, so $u_0$ is a free equivariant embedding. Finally, if $g-du_0\cdot du_0$ is not positive-definite, we can multiply $u_0$ by a small enough constant. The proof applying Claim~\ref{claim initial} is concluded.
	
	\textbf{Proof of Claim~\ref{claim initial}} To prove this claim, we show that (a) the claim holds for some dimension (of the target space); (b) for any integer $q_1 > s_n+n$, the claim with $ q=q_1$ implies that with $ q=q_1-1$. 
	
	(a) We want to prove this part for $ q = s_{2n} + 2n $. It follows from the Whitney's embedding theorem \cite{whitney1944} that there is a regular 
	smooth embedding $u_1: {M/\Gamma}\to\RR^{2n}$. 
	
	Composing $u_1$ with the map $ u_2: y\mapsto(y,y_1^2,...,y_iy_j,...,y_{2n}^2)$, $1\leq i\leq j\leq 2n$, we obtain a smooth embedding $u = u_2\circ u_1:\TT^{n}\to\RR^{s_{2n}+2n}$. Note that $u_2$ is free. It is known that the composition of any regular smooth map with a free map is still free\footnote{To see this, note that the composition of a diffeomorphism and a free map is still free\cite[Discussion after Definition~1.1.1, Section~1.1]{han2006isometric}, so freeness does not rely on choices of coordinates. Applying the inverse function theorem to the regular smooth map, its image can be locally realized as a smooth submanifold of the domain of the free map. A sketch of the proof is also available in \cite[Semisolutions of 3.1 and 3.5]{petrunnin2025threelectures}.}, so $u$ is free, and hence has linearly independent second derivatives.
	
	(b) Assume that $f$ is a smooth immersion $ {M/\Gamma}\to\RR^q $, with $q> s_n+n$, which has linearly independent second order partial derivatives. Denote by $\Sigma(x)$ the collection of all unit vectors in the subspace of $ \RR^q $ spanned by second order partial derivatives of $f$ at $x$; that is
	\begin{equation}
		\Sigma(x) = \{v||v|=1,v\in span\{\partial_i\partial_jf(x);i,j=1,...,n\}\},
	\end{equation}
	where every $ \partial_i\partial_jf(x) $ is regarded as a vector in $\RR^q$.
	
	Let $\Sigma$ be the union of all spheres $(x,\Sigma(x))\subset{M/\Gamma}\times\RR^q$ for all $x\in{M/\Gamma}$.\footnote{Equivalently, $\Sigma$ is the union of all fibres of the spherical bundle defined with base space ${M/\Gamma}$, total space $\mathop{\bigcup}\limits_{x\in {M/\Gamma}}(x,\Sigma(x) )$ and projection $ (x,v)\mapsto x $ for $v\in \Sigma(x) $. } Note that for any $x\in{M/\Gamma}$, the dimension of $ \Sigma(x) $ is $s_n-1$. Then $\Sigma$ defines an $(s_n+n-1)$-dimensional smooth submanifold of ${M/\Gamma}\times\RR^q$. 
	
	We want to carefully choose a projection from $\RR^q$ to some $q-1$-dimensional linear subspace such that its composition with $f$ still has linearly independent second order partial derivatives. Consider the projection ${\pi_0}$ from ${M/\Gamma}\times\RR^q$ to $\RR^q$ given by $(x,v)\mapsto v$. Since the dimension of $\Sigma$ is less than $q-1$, the image $ {\pi_0}(\Sigma) $ has measure zero in $\SS^{q-1}$, implying that $\SS^{q-1}\setminus {\pi_0}(\Sigma)$ is nonempty (so it contains at least one point). 
	
	Take any $ v \in  \SS^{q-1}\setminus {\pi_0}(\Sigma)$. Let $V$ be the $(q-1)$-linear subspace of $\RR^q$ normal to $v$. Denote by ${\pi_{v}}:\RR^q\to\RR^q$ the projection from $\RR^q$ to $V$. If we fix $x\in{M/\Gamma}$ and identify $ \Sigma(x) $ as a subset of the unit sphere in $\RR^q$ centered at the origin, it follows that ${\pi_{v}}$ is injective on $ \Sigma(x) $. Since ${\pi_{v}}$ is linear, it is also injective on the linear space $ span\{\partial_i\partial_jf(x);i,j=1,...,n\} $. Note that ${\pi_{v}}$ commutes with the partial derivative operator; that is, $ \partial_i\partial_j({\pi_{v}}\circ f) = {\pi_{v}}(\partial_i\partial_jf) $. It follows that for every $x$, the space $ span\{\partial_i\partial_j({\pi_{v}}\circ f)(x);i,j=1,...,n\} $ is also of dimension $s_n$, so the composition ${\pi_{v}}\circ f$ has linearly independent second order partial derivatives.
	
	It remains to show that the projection ${\pi_{v}}$ can be chosen such that ${\pi_{v}}\circ f$ is still an immersion. It suffices to show that for generic ${\pi_{v}}$, ${\pi_{v}}\circ f$ has linear independent first derivatives. 
	
	This can be done in a manner similar to above. Denote by $\Sigma'(x)$ the set
	\begin{equation}
		\Sigma'(x) = \{v||v|=1,v\in span\{\partial_if(x);i=1,...,n\}\},
	\end{equation}
	where every $\partial_if(x)$ is regarded as a vector in $\RR^q$. Let $\Sigma'$ be the union $ \mathop{\bigcup}\limits_{x\in {M/\Gamma}}(x,\Sigma'(x))\subset{M/\Gamma}\times\RR^q $. This is an $(n-1)$-dimensional smooth submanifold of ${M/\Gamma}\times\RR^q$. Recall our definition of the projection ${\pi_0}:{M/\Gamma}\times\RR^q \to \RR^q, (x,v)\mapsto v$. Then the image ${\pi_0}(\Sigma')$ has zero measure in $\SS^{q-1}$. Hence for any $v'\in\SS^{q-1}\setminus{\pi_0}(\Sigma')$, the projection ${\pi_{v'}}:\RR^q\to\RR^q$ from $\RR^q$ to the $(n-1)$-dimensional linear subspace of $\RR^q$ normal to $v'$ is injective on $\Sigma'(x)$ (when we identify $\Sigma'(x)$ as a subset of the unit sphere in $\RR^q$ centered at the origin). It follows that ${\pi_{v'}}\circ f$ has linear independent first derivatives and hence is an immersion.
	
	Finally, we note that both $ {\pi_0}(\Sigma) $ and ${\pi_0}(\Sigma')$ are of dimension smaller than $(q-1)$. Thus it suffices to choose a projection ${\pi_{v}}$ determined by any vector $v\in\SS^{q-1}\setminus\{{\pi_0}(\Sigma)\cup{\pi_0}(\Sigma')\}$. We conclude the proof of (b) by taking $ {w_0} = {\pi_{v}}\circ f :{M/\Gamma}\to\RR^{q-1} $.
\end{proof}

\subsection{Proof of Theorem~\ref{theorem main}}
\label{subsection main theorem}
To prove Theorem~\ref{theorem main}, we need two auxiliary lemmas. The first lemma (Lemma~\ref{lemma localizing}) reduces the embedding problem of $M$ to that of the pre-image of a simply-connected, bounded open subset of ${M/\Gamma}$ via the covering map. It is a modification of a theorem by Günther on general manifolds \cite[Definition~2.1, p.167 and Theorem~2.2, p.168]{Gunther1989}, translated into English as follows.

\begin{definition}[``Property~(E)", Günther]
	\label{definition property E}
	Let $U$ be an open subset of a smooth manifold $\mathcal{M}$. A smooth symmetric $2$-tensor field $h$ on $\mathcal{M}$ is said to \textit{be of Property~(E) on $U$} if it is possible to choose a proper local coordinate system on $U$ such that $ h = {a}^4dx_1^2 $ with a smooth function ${a}:\mathcal{M}\to\RR$ whose support is contained in $U$.
\end{definition}
\begin{prop}[Günther]
	\label{proposition gunther}
	Let $ (\mathcal{M},g) $ be a smooth manifold. There exists a locally finite open cover $\{U_l\}$ by simply-connected, relatively compact sets and at most countably many smooth $2$-tensor fields $h_l$ of Property~(E) on $U_l$ satisfying
	\begin{equation}
		\label{eq localize summation gunther's original proposition}
		g = \mathop{\sum}\limits_{l}h_l
	\end{equation}
\end{prop}
For the convenience of readers, we attach an English translation of Günther's proof of Proposition~\ref{proposition gunther} in Appendix~\ref{appendix tranlation}. 


Let us introduce some notation. Denote the covering map $M\to {{{{M/\Gamma}}}}$ by $p$. For any simply connected, open subset $U\subset{M/\Gamma}$, ${p^{-1}(U)}$ can be represented as a disjoint union of open subsets of $M$. Each of these open subsets is diffeomorphic to ${U}$ via $p^{-1}$. Let $\widehat{U}$ be one of these open subsets of $M$.


\begin{lemma}
	\label{lemma localizing}
	There exists a finite open cover $\{{U_l}\}$ of $ {M/\Gamma} $ by simply-connected, open sets ${U_l}$, such that
	the metric $g$ of $M$ can be rewritten as the sum of finitely many smooth symmetric $2$-tensors $\widetilde{h}_l$,
	\begin{equation}
		\label{eq localize summation}
		g = \mathop{\sum}\limits_{l}\widetilde{h}_l
	\end{equation}
	where $\widetilde{h}_l$ is $\Gamma$-periodic and of Property~(E) on each $\tau\widehat{U}_l,\tau\in\Gamma$. Moreover, it is obvious that $ \widehat{U}_l$ does not intersect any $ \tau{\widehat{U}_l} $ for nontrivial $\tau\in\Gamma$.
	
\end{lemma}
\begin{proof}
	Applying Proposition~\ref{proposition gunther} to ${M/\Gamma}$, we obtain a finite open cover $\{{U_l}\}, l=1,2,...,L$ of ${{{{M/\Gamma}}}}$ by simply connected open sets and smooth symmetric $2$-tensor fields $h_l$ of Property~(E) on ${U_l}$, such that $g=\mathop{\sum}\limits_{l}h_l$.
	
	
	For every $l$, we define $\widetilde{h}_l = h_l\circ p$. This is a $\Gamma$-periodic symmetric $2$-tensor field on $M$. Clearly, it is of Property~(E) on $\tau\widehat{U}_l$. The proof is concluded.

\end{proof}

The next lemma (Lemma~\ref{lemma technical tool by gunther}) is a key proposition proved by Günther \cite[Theorem~2.4, p.~169]{Gunther1989}. Given a smooth free embedding of $M$, it allows us to perturb the induced metric locally so that the image of the perturbed embedding stays arbitrarily close to that of the embedding itself with respect to the Euclidean distance in $\RR^q$. 
\begin{lemma}[Günther]
	\label{lemma technical tool by gunther}
	Let $\mathcal{M}$ be a smooth manifold. Let $U$ be a simply-connected, relatively compact open subset of $\mathcal{M}$ and $u_0:\mathcal{M}\to\RR^{q}$ a smooth free embedding with $q\geq s_n+n+5$. Let $h$ be a smooth symmetric $2$-tensor field of Property~(E) on $U$.
	
	Then for every $\epsilon>0$, there exists a smooth free embedding $u_{\epsilon}:\mathcal{M}\to\RR^{q}$, such that
	\begin{equation}
		du_{\epsilon}\cdot du_{\epsilon} = du_0 \cdot du_0 + h.
	\end{equation} 
	Moreover,
	\begin{enumerate}
		\item $u_{\epsilon} = u_0$ on $\mathcal{M}\setminus U$.
		\item  $|u_{\epsilon}(x) - u_0(x)|_{\RR^{q}}\leq \epsilon, \text{ for }x\in \mathcal{M}$.
		
	\end{enumerate}
\end{lemma}

We are now prepared to prove our main theorem.
\begin{proof}[Theorem~\ref{theorem main}]
	We prove (a) the existence of a free equivariant isometric immersion, and (b) an argument to avoid self-intersections.
	
	(a) This part is similar to Günther's original approach. We start with a free equivariant embedding $u_0:M\to\RR^q$ such that $g-du_0\cdot du_0$ is positive-definite. The existence of such an embedding is proved in Lemma~\ref{lemma initial}. Applying Lemma~\ref{lemma localizing} to $(M, g-du_0\cdot du_0)$, we obtain a finite open cover $\{{U_l}\}, l=1,2,...,L$ of $ {{{M/\Gamma}}} $ such that $g - du_0\cdot du_0 = \mathop{\sum}\limits_{l}\widetilde{h}_l$ where every $\widetilde{h}_l$ is of Property~(E) on every $\tau\widehat{U}_l, \tau\in\Gamma$.
		
	We want to construct a series of free equivariant embeddings $u_1,u_2,...,u_L:M\to\RR^q$ such that
	\begin{enumerate}
		\item $ du_l\cdot du_l = du_{l-1}\cdot du_{l-1} + \widetilde{h}_l$.
		\item $ |u_l(x)-u_{l-1}(x)|_{\RR^q}\leq \epsilon_l $. This inequality can be achieved for any $\epsilon_l>0$, and we choose $\epsilon_l$ later.
		\item $ u_l=u_{l-1} $ on $M\setminus {p^{-1}({U_l})} $.
	\end{enumerate}

	Let us start with $l=1$. Since the support of $\widetilde{h}_1$ is contained in $p^{-1}(U_l)$ and the latter can be represented as a disjoint union $ p^{-1}(U_1) =  \mathop{\sqcup}\limits_{\tau\in\Gamma}\tau{\widehat{U}_1} $, $ \widetilde{h}_1 $ is smooth and vanishing in a sufficiently small neighborhood of $\widehat{U}_1$. Hence it is possible to smoothly extend its restriction to $\widehat{U}_1$, $\widetilde{h}_1|_{\widehat{U}_1}$, to the entire $M$ such that $ \widetilde{h}_1|_{\widehat{U}_1}=0 $ for $M\setminus \widehat{U}_1$.
	
	Applying Lemma~\ref{lemma technical tool by gunther} to $\widehat{U}_1$ with $\widetilde{h}_1|_{\widehat{U}_1}$, we obtain that for every $\epsilon_1>0$, there exists a free embedding $\widehat{u}_1:M\to\RR^q$ such that 
	\[d\widehat{u}_1\cdot d\widehat{u}_1 = du_0 \cdot du_0 + \widetilde{h}_1|_{\widehat{U}_1}, \text{ for }x\in\widehat{U}_1.\]	
	Moreover $ |\widehat{u}_1-u_0|\leq \epsilon_1 $ and $\widehat{u}_1=u_0$ on $M\setminus\widehat{U}_1$. Since $\widehat{U}_1$ is diffeomorphic to ${U}_1$ (via $p|_{\widehat{U}_1}$), we can define $u_1:M\to\RR^q$ as
	\begin{equation}
		\label{u_1}
		u_1(x)=\begin{cases}
			&\widehat{u}_1\circ (p|_{\widehat{U}_1})^{-1} \circ p(x) , x\in{p^{-1}({U}_1)}\\
			&u_0(x), x\in M\setminus{p^{-1}({U}_1)}
		\end{cases} .
	\end{equation}
	Clearly, $u_1$ is an equivariant immersion. Since $\widehat{u}_1$ is free, so is $u_1$.
	
	Repeating the process above on $u_{l},l=1,...,L-1$, we obtain a series of free equivariant immersions $u_2,...,u_L$. It follows that $ du_L\cdot du_L = du_0\cdot du_0 + \mathop{\sum}\limits_{l}\widetilde{h}_l $ and $|u_L(x)-u_0(x)|\leq\mathop{\sum}\limits_{l:x\in{p^{-1}({U_l})}}\epsilon_l$. We want $u_L$ to be well-defined. That is, we want $|u_0-u_L|_{\RR^q}$ to be finite. Since $\widehat{U}_l$ does not intersect $\tau\widehat{U}_l$ for nontrivial $\tau\in\Gamma$, we have 
	\begin{equation}
		\mathop{\sum}\limits_{l:x\in{p^{-1}({U_l})}}\epsilon_l\leq \mathop{\sum}\limits_{l=1}^{L}\epsilon_l
	\end{equation}
	Here, the right hand side $\mathop{\sum}\limits_{l=1}^{L}\epsilon_l$ can be chosen to be finite. The proof of (a) is done.
	
	(b) We prove this part by choosing $\epsilon_l, l=1,2,...,L$ to be small enough. 
	
	Let us start with $l=1$. Recall that $u_{0}$ is an embedding. We show that it is possible to choose $\epsilon_1$ such that $u_1(M)$ has no self-intersection. We divide the proof into three cases.
	\begin{enumerate}
		\item $u_1(M\setminus p^{-1}({U}_1))$ has no self-intersection. 
		
		It follows from $u_1 = u_{0}$ on $M\setminus p^{-1}({U}_1)$.
		\item $u_1(p^{-1}({U}_1))$ has no self-intersection. 
		
		From Part (a), we know that $\widehat{u}_1:M\to\RR^q$ is an embedding, so $\widehat{u}_1(\widehat{U}_1)$ has no self-intersections. From the construction of $u_1$ in (a), $ u_1 = \widehat{u}_1\circ (p|_{\widehat{U}_1})^{-1} \circ p$ on $ p^{-1}({U}_1) $, so every $u_1(\tau\widehat{U}_1), \tau\in\Gamma$ does not have self-intersections. 
		
		It remains to show that $ u_1(\widehat{U}_1) $ does not intersect $ u_1(\tau\widehat{U}_1) $ for nontrivial $\tau\in\Gamma$. 
		
		
		Recall that our choice of ${U_1}$ (in Lemma~\ref{lemma localizing}) ensures that $ \min\limits_{\tau\in\Gamma}d_{\RR^q}(u_{0}(\widehat{U}_1), u_{0}(\tau\widehat{U}_1)) > 0 $. Take 
		\[\epsilon_1 \leq \frac{1}{2} \min\limits_{\tau\in\Gamma}d_{\RR^q}(u_{0}(\widehat{U}_1), u_{0}(\tau\widehat{U}_1))  .\] 
		
		It follows that every $u_1(\tau\widehat{U}_1)$ is contained in the $\epsilon_1$-neighborhood of $u_{0}(\tau\widehat{U}_1)$ in $\RR^q$.
		
		Since the $\epsilon_1$-neighborhoods of $u_{0}(\widehat{U}_1)$ and $u_{0}(\tau\widehat{U}_1)$ do not intersect, $ u_1(\widehat{U}_1) $ does not intersect $ u_1(\tau\widehat{U}_1)$.
		\item $ u_1(p^{-1}({U}_1)) $ does not intersect $ u_1(M\setminus p^{-1}({U}_1)) $. 
		
		It suffices to show that $u_1(\widehat{U}_1) $ does not intersect $ u_1(M\setminus p^{-1}({U}_1)) $. To see this, we note that $ u_1(M\setminus p^{-1}({U}_1)) = \widehat{u}_1(M\setminus p^{-1}({U}_1)) $ and $ u_1(\widehat{U}_1) = \widehat{u}_1(\widehat{U}_1) $. Since $\widehat{u}:M\to\RR^q$ is an embedding, the statement follows.
	\end{enumerate}
	By repeating this process inductively for $l=2,3,...,L$, the proof of (b) is concluded.
\end{proof}

\begin{appendices}
	\newpage
	\section{Günther's Proof of Proposition~\ref{proposition gunther}}
	\label{appendix tranlation}
	This is an English translation for the proof of Proposition~\ref{proposition gunther} \cite[Theorem~2.2, p.~168]{Gunther1989}. 
	
	
	\begin{theoremnonumber}[Günther]
		Let $ M^n $ be a smooth manifold equipped with metric ${U}$. There exists a locally finite open cover $\{U_l\}$ by simply-connected, relatively compact sets attached with at most countably many smooth $2$-tensor fields $h_l$ of Property~(E) on $U_l$ (Definition~\ref{definition property E}) satisfying
		\begin{equation}
			g = \mathop{\sum}\limits_{l}h_l
		\end{equation}
	\end{theoremnonumber}
	
	Let $V_\epsilon:=\{z\in \RR^n | |z|\leq \epsilon\}$ and $B:=V_1$. For the rest of this section, when we refer to a chart $(U, \varphi)$ on $M$, we implicitly assume that $\varphi(U) = V_\epsilon$, for some $\epsilon>0$. 
	
	A lemma \cite[Lemma~2.3, p.168]{Gunther1989} is needed.
	\begin{lemma}[Günther]
		\label{lemma gunther's tool lemma 2.3}
		Let \(g\in C^{r}(B,\RR^{D})\), and let \(F\subset C^{r}(B,\RR^{D})\) be a family of smooth maps with the following property:
		there exists a neighborhood \(W\subset\RR^{D}\) of \(g(0)\) such that every \(z\in W\) can be written as
		\[
		z=\alpha_{1}f_{1}(0)+\cdots+\alpha_{k}f_{k}(0),
		\]
		for suitable \(f_{1},\dots,f_{k}\in F\) and positive coefficients \(\alpha_{1},\dots,\alpha_{k}\in\RR\).
		Then there exist an open neighborhood \(B'\subset B\) of $0\in\RR^D$, positive functions
		\(a_{1},\dots,a_{\ell}\in C^{r}(B')\), and \(f_{1},\dots,f_{\ell}\in F\) such that
		\[
		g(x)=a_{1}(x)f_{1}(x)+\cdots+a_{\ell}(x)f_{\ell}(x)\qquad\text{for all }x\in B'.
		\]
	\end{lemma}
	
	\begin{proof}[Günther's theorem]
		Fix \(p\in M\). There is a relatively compact neighborhood \(U_{p}\subset M\) and a \(C^{r}\)-diffeomorphism
		\(\varphi_{p}\) from $U_p$ onto the open unit ball $B\in\RR^n$ with \(\varphi_{p}(p)=0\in\mathbb{R}^{n}\).
		Let $x_{1},\dots,x_{n}:U_p\to\RR$ be the coordinate functions given by \((U_{p},\varphi_{p})\) and let
		\begin{equation}
			L = \{f\in C^\infty(B)|f(x) = \mathop{\sum}\limits_{i=1}^na_ix_i,a_i\in\RR\}
		\end{equation}
		be the set of linear functions on $B$. Let $\chi\in C^{\infty}(M)$ be a function whose support is contained in $U_p$, and $f\in L$. Then the tensor field $h:=(\chi^4\partial_if\partial_jf)$ defines a symmetric $2$-tensor field of Property~(E) on $U_p$ (Definition~\ref{definition property E}).
		
		Let
		\begin{equation}
			F:=\{(\partial_if\partial_jf)_{1\leq i\leq j\leq n}|f\in L\}\subset C^\infty(B,\RR^{n(n+1)/2}).
		\end{equation}
		
		Let \(z=(z_{ij})_{1\le i\le j\le n}\in\mathbb{R}^{n(n+1)/2}\) be a vector such that $|z_{ij}-g_{ij}(0)|$ is sufficiently small so that 
		the $(n\times n)$-matrix \((z_{ij})\) is positive-definite. Hence there exist coefficients \(c_{ij}\in\mathbb{R}\) such that 
		\begin{equation}
			z_{ij} = \mathop{\sum}\limits_{l=1}\partial_if_l(0)\partial_jf_l(0), 
		\end{equation}
		where $f_l(x):=\mathop{\sum}\limits_{k=1}^nc_{kl}x_k, l=1,2,...,n$.
		
		Therefore the conditions of Lemma~\ref{lemma gunther's tool lemma 2.3} are satisfied by ${U}$ and $F$. Recall our notation: $V_r$ is the open ball of radius $r$ in $\RR^n$ centered at the origin. Applying Lemma~\ref{lemma gunther's tool lemma 2.3} to ${U}$ and $F$, we obtain \(\varepsilon_{p}>0\),
		\(\ell_{p}\in\mathbb{N}\), positive functions \(a_{1},\dots,a_{\ell_{p}}\in C^{r}(V_{\epsilon_p})\),
		and associated fields \(f_{1},\dots,f_{\ell_{p}}\in F\) such that on \(V_{\epsilon_p}\)
		\begin{equation}
			g_{ij}(x)=\mathop{\sum}\limits_{k=1}^{l_p}a_{k}^4(x)\partial_if_k(x)\partial_jf_k(x), x\in V_{\epsilon_p}.
		\end{equation}
		
		The family
		\begin{equation}
			\label{eq choices of U_l can be small}
			\mathcal{B}=\{\varphi_{p}^{-1}(V_{r}) : r\in(0,\epsilon_{p}],\ p\in M\},
		\end{equation}
		where $V_r$ is the open ball of radius $r$ in $\RR^n$ centered at the origin, defines a topological basis of \(M\).
		Choose a countably many, locally finite cover \(\{U_{l}\}_{l\geq 1}\subset\mathcal{B}\) by relatively compact sets and a partition of unity
		\(\{\chi_{l}^4\}_{l\geq 1}\) subordinate to it. As shown above, every \(\chi_{l}^4g\) is a finite sum of \(C^{r}\) tensor fields, each of which is of property~(E) (on its support)
		and has a support contained in \(U_{l}\). Since the cover $\{U_l\}$ is chosen to be locally finite, the sum
		\[
		\sum_{l}\chi_{l}^4g
		\]
		always has at most finitely many non-vanishing terms on $M$ and hence expresses \(g\) as a (at most countably many) sum of \(C^{r}\) tensor fields of Property~(E). This concludes the proof.
	\end{proof}
\end{appendices}

\newpage
\bibliography{D:/latex/ref}
\bibliographystyle{plainurl}
\end{document}